\documentclass[12pt]{amsart}
\usepackage[utf8]{inputenc}
\usepackage{geometry}
\usepackage[english]{babel}
\usepackage{amsthm}
\usepackage{amsmath}
\usepackage{MnSymbol}
\usepackage{enumerate}
\usepackage{amsfonts}
\usepackage{amscd}

\newtheorem{theorem}{Theorem}

\DeclareMathOperator{\Spec}{Spec}
\DeclareMathOperator{\Sing}{Sing}
\DeclareMathOperator{\Branch}{Branch}

\DeclareMathOperator{\loc}{loc}
\DeclareMathOperator{\sh}{sh}
\DeclareMathOperator{\et}{\text{\'{e}t}}

\DeclareMathOperator{\im}{im}
\DeclareMathOperator{\BL}{BL}
\DeclareMathOperator{\reg}{reg}

\newcommand{\ger}[1]{\mathfrak{#1}}
\newcommand{\imp}{\Rightarrow}
\newcommand{\bs}{\backslash}

\newtheorem{corollary}{Corollary}

\theoremstyle{definition}
\newtheorem*{definition}{Definition}
\newtheorem*{remark}{Remark}
\newtheorem*{note}{Note}
\newtheorem*{acknow}{Acknowledgement}
\newtheorem{example}{Example}
\newtheorem*{notation}{Notation}
\newtheorem{lem}{Lemma}

\begin{document}
\title{\'{E}tale Covers and Local Algebraic Fundamental Groups}

\begin{abstract}Let  $X$ be a normal noetherian scheme and $Z \subseteq X$ a closed subset of codimension $\geq 2$. We consider here the local obstructions to the map $\hat{\pi}_{1}(X\backslash Z) \to \hat{\pi}_{1}(X)$ being an isomorphism. Assuming $X$ has a regular alteration, we prove the equivalence of the obstructions being finite and the existence of a Galois quasi-\'{e}tale cover of $X$, where the corresponding map on fundamental groups is an isomorphism.
\end{abstract}
\date{}

\author{Charlie Stibitz}
\address{Department of Mathematics, Princeton University, Princeton, NJ 08544, USA}
\email{cstibitz@math.princeton.edu}
\maketitle 
\section{Introduction} Suppose that $X$ is a normal variety over $\mathbb{C}$ and $Z \subseteq X$ is a closed subset of codimension $2$ or more. Then a natural question to pose is whether the surjective map of fundamental groups $\pi_{1}(X\bs Z) \to \pi_{1}(X)$ is an isomorphism. For general normal schemes we can ask the same question for \'{e}tale fundamental groups.  For a regular scheme the Zariski-Nagata theorem on purity of the branch locus implies the above map on \'{e}tale fundamental groups is an isomorphism (see \cite{zariski1958purity}, \cite{nagata1959purity}). For a general normal scheme however this map need not be an isomorphism so that \'{e}tale covers of $X\bs Z$ need not extend to all of $X$. \\
\indent The next question to ask is what are the obstructions to the above map being an isomorphism. Restricting any cover to a neighborhood of a point, we see that in order for it to be \'{e}tale, it must restrict to an \'{e}tale cover locally. Hence each point of $Z$ gives rise to a possible obstruction determined by the image of the local \'{e}tale fundamental group into the \'{e}tale fundamental group of $X\bs Z$. Assuming these all vanish, the above map will be an isomorphism. \\ 
\indent Even if they do not vanish, we can still hope for something in the case where all the obstructions are finite. A first guess might be that this would imply that the kernel of the above map is finite, yet Example 1 of the singular Kummer surface  shows that this can be far from true in general. What is true however is that after a finite cover that is \'{e}tale in codimension 1 the corresponding map is an isomorphism. The main theorem here is that this is is fact equivalent to finiteness of the obstructions and a couple other similar conditions:  

\begin{theorem} Suppose that $X$ is a normal noetherian scheme of finite type over an excellent base $B$ of dimension $\leq 2$. Let $Z = \Sing(X) \subseteq X$. Then the following are equivalent. 
\begin{enumerate}[(i)]
    \item For every geometric point $x \in Z$ the image $G_{x} := \im[\pi_{1}^{\et}(X_{x}\bs Z_{x}) \to \pi^{\et}_{1}(X\bs Z)]$ is finite (see $\S 2$ for definitions of $X_{x}$ and $Z_{x}$).
    \item There exists a finite index closed normal subgroup $H \subseteq \pi_{1}^{\et}(X\bs Z)$ such that $G_{x} \cap H$ is trivial for every geometric point $x \in X$. 
    \item For every tower of quasi-\'{e}tale Galois covers of $X$
   \[X \leftarrow X_{1} \leftarrow X_{2} \leftarrow X_{3} \leftarrow \cdots\]
    $X_{i+1} \to X_{i}$ are \'{e}tale for $i$ sufficiently large.
   \item There exists a finite, Galois, quasi-\'{e}tale cover $Y \to X$ by a normal scheme $Y$ such that any \'{e}tale cover of $Y_{\reg}$ extends to an \'{e}tale cover of $Y$.
\end{enumerate}
\end{theorem}

\indent As hinted to above, we can rephrase the question of the $\pi_{1}^{\et}(X\bs Z) \to \pi_{1}^{\et}(X)$ being an isomorphism as a purity of the branch locus statement of $X$. Both are equivalent to the fact that any \'{e}tale cover of $X\bs Z$ extends to an \'{e}tale cover of $X$. By purity of the branch locus for $X_{\text{reg}}$, any \'{e}tale cover of $X\bs Z$ will at least extend to $X_{\text{reg}}$. Hence it is enough to consider the case where $Z = X_{\reg}$, as we have done in the theorem. From this point of view, (iv) says that we can obtain purity after a finite, Galois, quasi-\'{e}tale cover. \\
\indent The following example of the singular Kummer surface then elucidates what is going on in the above theorem. \\

\begin{example}
Consider the quotient $\pi:A \to A/\pm = X$, where $A$ is an abelian surface and $X$ is a singular Kummer surface over $\mathbb{C}$. Away from the 16 $2$-torsion points this map is \'{e}tale, but at the $2$-torsion points it ramifies. Each of these $2$-torsion points gives rise to a nontrivial $\mathbb{Z}/2\mathbb{Z}$ obstruction. In particular any \'{e}tale cover of $A$ will give a cover of $X$ not satisfying purity of the branch locus, and in particular there are infinitely many such covers. On the other hand the fundamental group of $X$ is trivial, which can be seen as $X$ will be diffeomorphic to a standard Kummer surface given as a singular nodal quartic in $\mathbb{P}^{3}$ with the maximum number of nodes. In particular its \'{e}tale fundamental group is trivial, and there are no \'{e}tale covers of $X$. Note that on the other hand as $A$ is smooth we  obtain purity on a finite cover of $X$ that is \'{e}tale away from a set of codimension $2$ as in (iv) of the above theorem. In this sense, although the kernel of the map $\hat{\pi}_{1}(X_{\reg}) \to \hat{\pi}_{1}(X)$ is large it is not far away from satisfying purity. 
\end{example}

\indent The recent history of studying these two problems started with a paper of Xu \cite{X14} showing that the local obstructions are finite for all klt singularities over $\mathbb{C}$. From this Greb, Kebekus, and Peternell \cite{GPK16} were able to show the global statements of (iii) and (iv) in the above theorem for klt singularities. Their proof used essentially the existence of a Whitney stratification, which allowed them to check only finitely many strata to prove a cover is \'{e}tale. Then in positive characteristic, Caravajal-Rojas, Schwede, and Tucker \cite{CST16} proved again that the local obstructions are finite for strongly $F$-regular singularities (which are considered a close analogue of klt singularities in positive characteristic). Using a bound on the size of the local fundamental groups from this paper, Bhatt, Carvajal-Rojas, Graf, Tucker and Schwede \cite{bhatt16} were able to construct a stratification that enabled them to run a similar local to global argument to deduce statements of the form (ii), (iii), and (iv) for strongly $F$-regular varieties. It is also worth noting in the recent preprint of Bhatt, Gabber, and Olsson \cite{bhatt2017} they are able to reprove the results in characteristic $0$ by spreading out to characteristic $p$. \\
\indent The proof of Theorem 1 can be broken down essentially into two parts. One is the construction of a stratification that allows us to deal with only finitely many obstructions. The second is a completely group theoretic fact about profinite groups, namely if we have a finite collection of finite subgroups of a profinite group, then there exists a closed finite index subgroup which intersects all of these groups trivially. Note that the assumption that the covers are Galois and the finite index subgroup is normal is essential. In fact due to the choice of basepoint that we have suppressed above (note that $x$ is \textit{not} the basepoint of these fundamental groups), $G_{x}$ is actually only a conjugacy class of a finite subgroup inside of the group $\pi_{1}^{\et}(X\bs Z)$. \\
\indent Finally it is worth noting that in Xu's paper a different local fundamental group, $\pi_{1}^{\et}(X_{x}\bs \{x\})$ was shown to be finite. If we instead defined our obstruction groups $G_{x}$ as Xu did then in fact the theorem is false, as will be shown in the following example. There are two ways around this for klt singularities: either showing the larger fundamental groups $\pi_{1}^{\et}(X_{x}\bs Z_{x})$ are finite in the klt case, or show that a similar implication (i) $\imp$ (ii),(iii) will hold as long as the smaller local fundamental groups of all covers \'{e}tale in codimension 1 remain finite, which is the case for klt singularities.  We will discuss this issue further in the last section. \\

\begin{example}
 Consider  $X = CS$ the cone over a Kummer surface $S = A/\pm $ where $A$ is an abelian surface. Then there are three types of singular points: \\
 \indent First consider the case where $x$ is the generic point of the cone over one of the nodes. Then $X_{x}$ has a regular double quasi-\'{e}tale cover ramifying at $x$ (note that since we have localized there is no difference between the two possible fundamental groups). This shows $\hat{\pi}_{1}^{\loc}(X_{x}\bs\{x\}) = \hat{\pi}_{1}^{\loc}(X_{x}\bs Z_{x}) \cong \mathbb{Z}/2\mathbb{Z}$, and there is no ambiguity in which definition we choose. In general this will work for the generic point of any irreducible component of the singular locus. \\
 \indent The next type of point where we start to see a difference is when  $x$ is a closed point in the cone over a node of $S$. Then in this case $\hat{\pi}_{1}^{\loc}(X_{x}\bs\{x\})$ is trivial while $\hat{\pi}_{1}^{\loc}(X_{x}\bs \Sing(X)_{x}) \cong \mathbb{Z}/2\mathbb{Z}$. Although they are different they are at least both finite. On the other hand if we desire for these groups to behave well under specialization it is clear that $\hat{\pi}_{1}^{\loc}(X_{x}\bs \Sing(X)_{x})$ is the better choice. \\
\indent The last type of point, where the real problem occurs, is the cone point $x \in CS$.  First consider the fundamental group $\hat{\pi}_{1}^{\loc}(X_{x}\bs\{x\})$. Then this will be isomorphic to $\hat{\pi}_{1}(S) \cong 0$ by the Lefschetz hyperplane theorem. In particular all the local fundamental groups defined in this sense are finite. So if this version of the theorem were true then any tower as above would stabilize. On the other hand $\hat{\pi}_{1}(S_{\reg})$ is infinite, giving an infinite tower of cones $X=CS \leftarrow CS_{1} \leftarrow CS_{2} \leftarrow \cdots$ Galois over $X$ and quasi-\'{e}tale. In particular finiteness of all the local fundamental groups $\hat{\pi}_{1}^{\loc}(X_{x}\bs \{x\})$ does not imply finiteness of the local fundamental groups $\hat{\pi}_{1}^{\loc}(X_{x}\bs\Sing(X)_{x})$. Note that in this case the singularity at the origin is not klt.
\end{example}

\begin{notation}
Given a finite morphism $f:Y \to X$ the branch locus, written $\Branch(f)$, is the locus over which $f$ fails to be \'{e}tale. A finite morphism $f:Y \to X$ is quasi-\'{e}tale if it is \'{e}tale in codimension $1$ or in other words the branch locus has codimension $\geq 2$. By purity of the branch locus for normal schemes quasi-\'{e}tale is equivalent to being \'{e}tale over the regular locus.
\end{notation}

\begin{acknow} I would like to thank my advisor J\'{a}nos Koll\'{a}r for his constant support. Also I would like to thank Ziquan Zhuang for several useful discussions. 
\end{acknow}

\section{Local Fundamental Groups}

\indent In this section we review some basic facts and definitions about the fundamental groups  we will be considering. 

\begin{definition} Suppose that $(R,\ger{m})$ is a strictly Henselian local normal domain and that $Z \subseteq \Spec(R)$ is a closed subset of codimension $\geq 2$. Then we define the algebraic local fundamental with respect to $Z$ to be $\pi_{1}^{\et}(\Spec(R)\bs Z)$. If $x$ is a normal geometric point of an irreducible scheme $X$ and $Z \subseteq X$ has codimension $\geq 2$, we define the local space $X_{x} = \Spec(\mathcal{O}_{X,x}^{\sh})$, the spectrum of the strict Henselization of the local ring. This comes with a map $\iota:X_{x} \to X$ and we define $Z_{x} = \iota^{-1}(Z)$. We then define the local fundamental group at $x$ with respect to $Z$ to be the algebraic local fundamental group of the strict Henselization of  $\mathcal{O}_{X,x}$ with respect to the closed set $Z$ and use the notation $\hat{\pi}_{1}^{\loc}(X_{x}\bs Z_{x})$. For any geometric point $x \in Z$ we define $G_{x} := \im[\hat{\pi}_{1}^{\loc}(X_{x}\bs Z_{x}) \to \hat{\pi}_{1}(X\bs Z)]$, the obstructions occurring in theorem 1. 
\end{definition}

\begin{note}
The definition above depends on a choice of both strict Henselization (requiring a choice of separable closure of the residue field of $x$) and a choice of base point. In particular the choice of base point implies that the groups $G_{x}$ are only defined up to conjugacy. It is for this reason that we need to take Galois morphisms in the main theorem. For counterexamples when the morphisms are not Galois see \cite{GPK16}. Also note the difference between the definition here and that used in \cite{X14}. 
\end{note} 

\indent The following basic lemma shows the purpose in using Henselizations when defining the local fundamental group. \\

\begin{lem} [\cite{bhatt16} Claim 3.5] Suppose that $f:Y \to X$ is a quasi-\'{e}tale morphism of normal schemes, that is \'{e}tale away from some subset $Z$ of codimension $\geq 2$. Then $f$ is \'{e}tale over a geometric point $x \in X$ if and only if the pull back of the map to $U_{x}: = X{x}\bs Z_{x}$ is trivial. 
\end{lem}
\begin{proof}
It is enough to prove that the map is \'{e}tale once we pull back to the strict Henselization of the local ring. Now in this case if the map $f$ is \'{e}tale then it induces a trivial cover of $\Spec(\mathcal{O}_{X,x}^{sh})$ and hence of the open set $V$. On the other hand if the cover of $U_{x}$ is trivial, then so is the cover of $\Spec(\mathcal{O}_{X,x}^{sh})$ since the varieties are normal. Hence  the morphism is \'{e}tale. 
\end{proof}

\section{The Branch Locus of a Quasi-\'{E}tale Morphism}

\indent In this section we consider the question of where a quasi-\'{e}tale cover of a normal variety $X$ branches. We will see that there are finitely many locally closed subsets of $X$ such that any branch locus is the union of some subcollection of these subsets. Moreover we will show that such a stratification is possible to compute in terms of any alteration. We start with the criterion for telling if a morphism is \'{e}tale from an alteration. \\

\noindent \textbf{Alterations.} \cite{dJ96} A map $\pi:\tilde{X} \to X$ is an \textit{alteration} if it is proper, dominant and generically finite. A regular alteration will be an alteration where $\tilde{X}$ is regular. In his work, de Jong showed that regular alterations exist for noetherian schemes of finite type over an excellent base of dimension $\leq 2$. We will say that a divisor $E \subset \tilde{X}$ is \textit{exceptional} if $\pi(E)$ has codimension at least $2$ on $X$. \\

\begin{lem} Suppose that $X$ is normal Noetherian scheme, with a regular alteration $\pi:\tilde{X} \to X$. Let $f:Y \to X$ be a finite morphism of Noetherian schemes, and denote by $\tilde{f}:\tilde{Y} \to \tilde{X}$ the normalized fiber product of the maps. Then  $f$ is \'{e}tale if and only if $\tilde{f}$ is \'{e}tale and for any geometric point $x \in X$, $\tilde{f}$ induces a  trivial cover of $\pi^{-1}(x)$. 
\end{lem}
\begin{proof}
First suppose that $f:Y \to X$ is \'{e}tale. Then  the base change $Y \times_{X} \tilde{X} \to \tilde{X}$ is \'{e}tale, so that the fiber product was already normal. Hence it follows $\tilde{f}$ is \'{e}tale. Then since $\tilde{Y}$ is just the fiber product and $f$ is \'{e}tale,  for any of the $d$ points $q \in Y$ mapping to $p \in X$ we see that $\sigma^{-1}(q)  \cong \pi^{-1}(p) \times_{k(p)} k(q)$. Therefore $\tilde{f}$ is just the trivial degree $d$ cover on every fiber $\pi^{-1}(p)$. \\
\indent Now suppose that $\tilde{f}$ is \'{e}tale and induces a trivial cover on every fiber of $\pi$. Then in particular for any $x \in X$,  $f^{-1}(x)$ will have $\deg(f)$ geometric connected components. In particular it must be \'{e}tale at $x$ (\cite{mumfordoda}, V.7).
\end{proof}

\begin{example}
Each of the two conditions in the lemma above are easily seen to be necessary.  For example we can let $X$ be the cone over a smooth conic. This has a quasi-\'{e}tale double cover $f:\mathbb{A}^{2} \to X$. Blowing up the origins gives a map of the normalized fiber-products $\BL_{0}\mathbb{A}^{2} \to \BL_{0}X$, that will ramify along the exceptional divisors.  \\
\indent On the other hand we can take the $X$ to be the cone over an elliptic curve $E$. Take an \'{e}tale cover $E' \to E$, which will induce a quasi-\'{e}tale cover $X' \to X$ of their cones. After blowing up the origins we obtain the map of normalized fiber products which is \'{e}tale, but induces a nontrivial cover of the exceptional divisors. 
\end{example}

\indent Using the above lemma we can check whether $f$ is \'{e}tale based on a single regular alteration. Our next goal will be to show that based on this alteration we really only need to check that $f$ is \'{e}tale at finitely many points. To identify what are the points we need to check we require the following condition, which roughly says that the reduced fibers of a morphism fit together in a flat family.  \\

\noindent \textbf{Condition $\ast$.} Suppose that $g:Z \to S$ is a proper  morphism of Noetherian schemes with $S$ integral. Then $g$ satisfies this condition if there exists a purely inseparable morphism $i:S' \to S$ such that if $Z' = Z\times_{S} S' \to S'$ is the base change, then $Z_{\text{red}}' \to S'$ is flat with geometrically reduced fibers. \\

\indent We now show that there exists a stratification of $X$ such that $\pi$ will satisfy the above  condition $\ast$ over each of the strata. \\

\begin{lem} Let $\pi:\tilde{X} \to X$ be a morphism of Noetherian schemes. Then there exists a stratification $X = \bigcup S_{i}$, where the $S_{i}$ are irreducible locally closed subsets,  such that $\pi^{-1}(S_{i}) \to S_{i}$ satisfies condition ($\ast$) for all $i$. 
\end{lem}
\begin{proof}
We will proceed by Noetherian induction on $X$. Take an irreducible component $S$  of $X$. Consider the map $\pi^{-1}(S)_{\text{red}} \to S$. Taking an irreducible component  $W$ of $\pi^{-1}(S)_{\text{red}}$  if $W \to S$ is not separable, we can take some high enough power of the Frobenius so that the pullback by the map is separable. Doing this for every irreducible component of $\pi^{-1}(S)_{\text{red}}$ we may assume that the general fiber is reduced. Then taking an open subset $U$ of $S$ we may assume that every fiber of $\pi^{-1}(U)_{\text{red}} \to U$ is reduced and that this morphism is flat. Continuing on will give the desired stratification.
\end{proof}

\begin{lem}  [e.g. \cite{EGAIII} 7.8.6] Suppose that $g:Z \to S$ is a  morphism of Noetherian schemes satisfying condition ($\ast$) with $S$ integral. Then the number of connected components of geometric fibers are constant. 
\end{lem}
\begin{proof}
We have a purely inseparable morphism $S' \to S$ such that $Z' \to S'$ is flat with geometrically reduced fibers. Since $S' \to S$ is a universal homeomorphism, it follows that $Z' \to Z$ is a homeomorphism. Hence the number of connected components remains the same, so we can assume from the beginning that $Z \to S$ is flat with geometrically reduced fibers. \\
\indent Now in this case we will show that the Stein factorization of $g:Z \to S$ factors as $Z \to \hat{S} \to S$ where $\hat{S} \to S$ is \'{e}tale. Taking the strict Henselization of the local ring at any point we can reduce to the case where $S$ is the spectrum of a strictly Henselian local ring. In this case $\hat{S}$ is a product of finitely many local rings. Our goal is to show that these are isomorphic to $S$. Now consider a connected component $W$ of $Z$, so that the map $g:W \to S$ is flat and proper, with geometrically reduced fibers. Now since $W$ is connected and $S$ is the spectrum of strictly Henselian ring, the special fiber $W_{0}$ is also connected. But then since $W_{0}$ is reduced $H^{0}(W_{0},\mathcal{O}_{W_{0}}) = k(0)$. Hence we see by the theorem of Grauert that $\mathcal{O}_{S} \to g_{\ast}\mathcal{O}_{W}$ is an isomorphism. This implies that $\hat{S} \to S$ is thus \'{e}tale, so in particular the number of connected components of the geometric fibers are constant.
\end{proof}

\begin{theorem}Suppose that $X$ is a normal Noetherian scheme and  $\pi:\tilde{X} \to X$ a regular alteration. Then there exists a stratification $X = \bigcup_{i\in I}Z_{i}$ into locally closed subsets such that for any $f:Y \to X$ quasi-\'{e}tale, with $Y$ a normal Noetherian scheme,  $\Branch(f) = \bigcup_{i \in J \subset I}Z_{i}$. 
\end{theorem}
\begin{proof}
The above lemma gives a stratification $X = \bigcup_{i} S_{i}$ such that $\pi^{-1}(S_{i}) \to S_{i}$ satisfies condition $\ast$. Moreover we a finite number of exceptional divisors $E_{i}$ giving closed subsets $\pi(E_{i})$ on $X$. Putting these together gives our desired stratification of $X$. Our goal is then to show that any branch locus of a quasi-\'{e}tale morphism is a union of these strata. \\
\indent Consider $\tilde{Y} = (\tilde{X} \times_{X} Y)^{n}$ the normalized fiber product which comes with a morphism $\tilde{f}:\tilde{Y} \to \tilde{X}$ that is \'{e}tale away from the exceptional locus. Now by purity of the branch locus  $\Branch(\tilde{f}) = \bigcup_{i} E_{i}$ where the $E_{i}$ are some subset of the exceptional divisors. In particular the branch locus of $f$ will include $B=\bigcup_{i} \pi(E_{i})$, which will be a union of some strata. Now looking on the complement of $B$, and replacing $X$ by $X\bs B$ we can assume that $\tilde{f}$ is in fact \'{e}tale. In particular  $\tilde{f}^{-1}(\pi^{-1}(S_{i})) \to \pi^{-1}(S_{i}) \to S_{i}$ satisfies condition $\ast$. Hence  the number of connected components of the fibers are constant. This implies that for any point $s \in S_{i}$ that if the cover of $\pi^{-1}(s)$ is geometrically trivial, then the corresponding cover for other point in $S_{i}$ is also trivial. Hence we see that the branch locus must be a union of the strata. 
\end{proof}

\begin{remark} In the proof of (i) implying (iii) of the main theorem it  would be nice to apply this theorem directly on $X$. However when we take the normalized pullback of an alteration we may not get another alteration. To remedy this we will need to take alterations of varieties that are further along in the tower. 
\end{remark}

\section{Proof of the Main Theorem}

\indent In this section we prove the different implications in the main theorem.   \\

\noindent \textbf{(i) $\imp$ (ii).}
\begin{proof}
 Consider a regular alteration $\pi:\hat{X} \to X$. This will give us a stratification $X = \bigcup_{i} Z_{i}$. Now for each of the finitely many generic points $\eta_{i}$ of the different strata consider the finitely many finite groups $G_{i} = G_{\eta_{i}}$. Then as $\pi_{1}^{\et}(U)$ is profinite there exists some finite index closed normal subgroup $H$ intersecting all of these $G_{i}$ trivially. This corresponds to a quasi-\'{e}tale cover $\gamma:Y \to X$ that is \'{e}tale over $U$. Moreover by our choice of stratification for any geometric point $x$ we will also have that $G_{x} \cap H$ is trivial. Hence such a finite index normal subgroup $H$ can be taken uniformly for all $x \in X$. 
\end{proof}

\noindent \textbf{(ii) $\imp$ (i).} \\
\begin{proof}
 Our assumption (ii) gives a closed finite index normal subgroup $H \subseteq \pi_{1}^{\et}(U)$ such that $G_{x} \cap H = \{1\}$. Then in particular $G_{x} \cong G_{x}/G_{x} \cap H \subseteq \pi_{1}^{\et}(U)/H$ which is finite. Hence $G_{x}$ is finite as well.
\end{proof}

\noindent \textbf{(i) $\imp$ (iii).} \\
\begin{proof}
We  proceed by Noetherian induction. Consider our tower of finite morphisms denoted by $\gamma_{k}:X_{k+1} \to X_{k}$, and consider the collection $\mathcal{U}$ of open sets $U \subseteq X$ such that when we restrict the tower over $U$ the morphisms are eventually \'{e}tale. The assumption that all the morphisms are quasi-\'{e}tale implies that $X_{\reg} \in \mathcal{U}$. Since $X$ is assumed to be Noetherian this collection has a maximal element and our goal is to show that this must be all of $X$. \\
\indent Therefore we need to show that if $U \in\mathcal{U}$ and $U \neq X$ then we can find a larger $U' \in \mathcal{U}$. To do this take any $x$ a generic point of an irreducible component of $X\bs U$. Consider $X_{x} = \Spec(\mathcal{O}_{X,\eta}^{\sh})$ and restrict the tower of $X_{i}$ over $X_{x}$ to get a tower
\[\Spec(\mathcal{O}_{X,\eta}^{\sh}) = X_{x,0} \leftarrow X_{x,1} \leftarrow X_{x,2} \leftarrow X_{x,3} \leftarrow \cdots\]
 Now using the assumption (i) applied to the point $x$, it follows that eventually the covers will be trivial when restricted over the regular locus and hence will be \'{e}tale. This then shows that there exists some $N >> 0$ such that $\gamma_{n}$ is \'{e}tale over $\eta$ for $n \geq N$ and they are \'{e}tale over the open set $U$ coming from Noetherian induction.\\
\indent Now take a regular alteration  $\pi:\hat{X}_{N} \to X_{N}$. Then using $\pi$ we construct a stratification $X_{n} = \bigcup_{i}Z_{i}$ as before. Then any of the maps $X_{N+k}  \to  X_{N}$ must be \'{e}tale over $U$ and $\eta$. But because the branch locus must be a union of strata it follows that  these are all \'{e}tale over some open set $U' \supset U$ with $U' \ni \eta$. Hence such a larger $U' \in \mathcal{U}$ exists and by Noetherian induction we see that $X \in \mathcal{U}$. This proves property (iii).
\end{proof}

\noindent \textbf{(iii) $\imp$ (iv).} 
\begin{proof}
Assuming that no such cover exists, we inductively construct a tower $X \leftarrow X_{1} \leftarrow X_{2} \leftarrow X_{3} \leftarrow \cdots$ as in (iii) of the main theorem using Galois closures, such that none of the $X_{i+1} \to X_{i}$ are \'{e}tale. This will contradict our assumption, so eventually  every \'{e}tale cover of one of the $X_{i,\reg}$ will extend to an \'{e}tale cover of $X_{i}$. This gives the desired cover satisfying purity. 
\end{proof}

\noindent \textbf{(iv) $\imp$ (i).}
\begin{proof}
Consider a geometric point $x$ of  $X$. Take a cover $f:Y \to X$ as in (iv), and a geometric point $y$ of $Y$ mapping to $x$. Denote by $U$ the regular locus of $X$ and $Z = X\bs U$ the singular locus.  This gives rise to the following commutative diagram of fundamental groups. 
\[\begin{CD}
\hat{\pi}_{1}^{\loc}(Y_{y}\bs f^{-1}(Z)_{y}) @>>> \hat{\pi}_{1}(f^{-1}(U)) \\ 
@VVV @VVV \\
\hat{\pi}_{1}^{\loc}(X_{x} \bs Z_{x}) @>>> \hat{\pi}_{1}(U)
\end{CD}\]
Now the assumption on $Y$ implies that the top map is zero. On the other hand, the image of the map on the left is a finite index normal subgroup. Hence looking at the images in $\hat{\pi}_{1}(U)$, we see that $G_{x}$ has a trivial finite index subgroup and hence must be finite. 
\end{proof}

\section{Applications}

\indent Using our main theorem we can recover the results of \cite{GPK16} and \cite{bhatt16}. \\

\begin{corollary}
Suppose that $X$ is a normal klt variety over $\mathbb{C}$. Then $X$ satisfies the condition (ii). 
\end{corollary}
\begin{proof}
We want to show that $X$ satisfies condition (i). There are two issues to deal with if we wish to apply Xu's result \cite{X14}. First is the problem that in this paper the local fundamental groups are defined in terms of links instead of the local spaces $X_{x}$ given by Henselization. The second is that Xu proves the finiteness of $\hat{\pi}_{1}^{\loc}(X_{x}\bs\{x\})$ and we saw that this is not enough to guarantee (ii) in general. \\
\indent There are two ways to get around this. The first is to strengthen the result of Xu to prove the finiteness of algebraic local fundamental groups as considered in this paper. In the proof of his main theorem, Xu cuts down to a surface. It is then possible to consider only quasi-\'{e}tale covers instead of \'{e}tale covers of $X_{x}\bs \{x\}$, as these will agree after cutting down. Also you would need an equivalence of the algebraic local fundamental group defined in terms of links and Henselizations. Once this is done though (ii) will follow immediately from the main theorem. Note that also the recent result of \cite{bhatt2017} is strong enough to apply directly. \\
\indent The second way to prove this is to note that we can get around the issue of which fundamental group we consider when we work in a class of normal varieties $\mathcal{R}$ satisfying the following. We want for every $X \in \mathcal{R}$, and every quasi-\'{e}tale cover $Y \to X$ that $Y \in \mathcal{R}$, and also for every $x$ a geometric point of $X \in \mathcal{R}$ that $\hat{\pi}_{1}^{\loc}(X_{x}\bs{x}\}$ is finite. In particular klt singularities satisfy both these conditions by \cite{X14}. Then under these assumptions, the same argument for (i) implies (ii) works with the fundamental groups $\hat{\pi}_{1}^{\loc}(X_{x}\bs{x}\}$. This approach is used in \cite{GPK16}.
\end{proof}

\begin{corollary}
Suppose that $X$ is a normal $F$-finite strongly $F$-regular variety over a field of characteristic $p$. Then $X$ satisfies the condition (ii).
\end{corollary}
\begin{proof}
In this case the result of Carvajal-Rojas, Schwede, and Tucker \cite{CST16} applies directly to the main theorem without any changes. 
\end{proof}

\bibliography{etalecovers}{}

\begin{thebibliography}{10}

\bibitem{bhatt16}
Bhargav Bhatt, Javier Carvajal-Rojas, Patrick Graf, Karl Schwede, and Kevin
  Tucker.
\newblock {\'E}tale fundamental groups of strongly {$F$}-regular schemes.
\newblock 2016, arXiv:1611.03884.

\bibitem{bhatt2017}
Bhargav Bhatt, Ofer Gabber, and Martin Olsson.
\newblock Finiteness of {\'e}tale fundamental groups by reduction modulo {$p$}.
\newblock 2017, arXiv:1705.07303.

\bibitem{CST16}
Javier Carvajal-Rojas, Karl Schwede, and Kevin Tucker.
\newblock Fundamental groups of {$F$}-regular singularities via
  {$F$}-signature.
\newblock 2016, arXiv:1606.04088.

\bibitem{dJ96}
A~Johan de~Jong.
\newblock Smoothness, semi-stability and alterations.
\newblock {\em Publications Math{\'e}matiques de l'Institut des Hautes
  {\'E}tudes Scientifiques}, 83(1):51--93, 1996.

\bibitem{GPK16}
Daniel Greb, Stefan Kebekus, Thomas Peternell, et~al.
\newblock {\'E}tale fundamental groups of kawamata log terminal spaces, flat
  sheaves, and quotients of abelian varieties.
\newblock {\em Duke Mathematical Journal}, 165(10):1965--2004, 2016.

\bibitem{EGAIII}
Alexander Grothendieck.
\newblock El{\'e}ments de g{\'e}om{\'e}trie alg{\'e}brique (r{\'e}dig{\'e}s
  avec la collaboration de jean dieudonn{\'e}): {III}. {E}tude cohomologique
  des faisceaux coh{\'e}rents, premiere partie.
\newblock {\em Publications Math{\'e}matiques de l'IHES}, 11:5--167, 1961.

\bibitem{mumfordoda}
David Mumford and Tadao Oda.
\newblock Algebraic geometry: {II}, 2015.

\bibitem{nagata1959purity}
Masayoshi Nagata et~al.
\newblock On the purity of branch loci in regular local rings.
\newblock {\em Illinois Journal of Mathematics}, 3(3):328--333, 1959.

\bibitem{X14}
Chenyang Xu.
\newblock Finiteness of algebraic fundamental groups.
\newblock {\em Compositio Mathematica}, 150(03):409--414, 2014.

\bibitem{zariski1958purity}
Oscar Zariski.
\newblock On the purity of the branch locus of algebraic functions.
\newblock {\em Proceedings of the National Academy of Sciences},
  44(8):791--796, 1958.

\end{thebibliography}
\bibliographystyle{hplain}

\end{document}